\documentclass[10pt]{article}
\usepackage[utf8]{inputenc}

\usepackage{amssymb,amsthm,amsmath}
\usepackage{mathrsfs}
\usepackage{enumerate}
\usepackage{graphicx,xcolor}
\usepackage{setspace}
\usepackage[hidelinks]{hyperref}

\usepackage[scr=boondox,  
            cal=esstix]   
           {mathalpha}

\newcommand{\dd}{\mathrm{d}}
\newcommand{\E}{\mathbb{E}}

\newcommand{\1}{\textbf{1}}
\newcommand{\R}{\mathbb{R}}

\newcommand{\cE}{\mathscr{E}}
\newcommand{\hh}{\mathscr{h}}

\DeclareMathOperator{\Var}{Var}
\DeclareMathOperator{\sgn}{sgn}

\newtheorem{theorem}{Theorem}
\newtheorem{lemma}[theorem]{Lemma}

\theoremstyle{remark}
\newtheorem{remark}[theorem]{Remark}

\theoremstyle{definition}

\title{\vspace{-3em}
Moments of sums of exponentials, beyond CHS
}

\author{Silouanos Brazitikos$^\dagger$, \ Colin Tang$^*$, \
Tomasz Tkocz\footnote{Email: ttkocz@math.cmu.edu. Research supported in part by NSF grant DMS-2246484.}
}

\date{\begin{normalsize}
\emph{$^\dagger$Department of Mathematics and Applied Mathematics, University of Crete, 70013 Heraklion - Crete, Greece.} \\\vspace*{0.7em}
\emph{$^*$Department of Mathematical Sciences, Carnegie Mellon University, Pittsburgh, PA 15213, USA}
\end{normalsize}}

\begin{document}

\maketitle

\begin{abstract}
We establish a sharp lower bound on the $L_p$-norm of sums of independent exponential random variables with fixed variance, for $p \geq 2$, thus extending Hunter's positivity theorem (1976) for completely homogeneous polynomials. We determine the exact regime of $p$ where such sums enjoy Schur-monotonicity.
\end{abstract}

\bigskip

\begin{footnotesize}

\end{footnotesize}

\bigskip

\section{Introduction}

The importance of exponential distribution in probability theory perhaps cannot be overstated. This is deeply engraved in its memoryless property (uniquely characterising it). As it naturally models the waiting time between independent random events occurring at a constant rate, the exponential distribution intrinsically connects to the Poisson process, making it essential for describing event arrival patterns, allowing for tractable analysis in survival modelling, reliability theory, and queueing systems, just to name a few concrete areas of its broad theoretical and practical applicability (see, e.g. \cite{Bal, King}). This is also related to the exponential distribution serving  as a building block for more complex distributions (say the Gamma, or Weibull ones), as well as continuous-time stochastic processes.

\subsection{Geometric motivation: simplex slicing}
Our main motivation stems from an inherent connection of the exponential distribution to the uniform measure on a simplex -- see Theorems 2.1 and 2.2 in Chapter 5 of \cite{Luc} -- and, more specifically, the study of extremal volume sections of the regular simplex, initiated in \cite{Webb} followed up in \cite{Brz}, quite recently reinvigorated in \cite{MRTT, MTTT}. In essence, the volume of such sections can be expressed probabilistically in terms of the density of sums of independent exponential random variables, for details see \cite{Webb}, or (1) in \cite{MRTT}. This point of view has naturally prompted questions about sharp $L_p$ bounds for sums of independent random variables, with the singular limiting case $p \downarrow -1$ corresponding to volume. We refer to the survey \cite{NT-surv} showcasing this paradigm, as well as to \cite{CKT, CST, MRTT} for results along those lines providing probabilistic extensions to Ball's cube-slicing \cite{Ball-cube}, the Oleszkiewicz-Pe\l czy\'nski polydisc slicing \cite{OP}, and Webb's simplex-slicing \cite{Webb}, respectively. 

On the other hand, classically, sharp $L_p$ bounds have been well-studied for positive $p$ in probability theory, the body of results often referred to as \emph{Khinchin-type} inequalities (the term coined after Khinchin's classical applications of those to his study on the law of the iterated logarithm \cite{Khi}). Khinchin-type inequalities concern simply reversals to H\"older type inequalities, up to multiplicative constants. Namely, one seeks sharp \emph{upper} bounds on the $L_q$-norm by the $L_p$-norm for $p < q$, thus reversing the generic standard H\"older's bound $\|\cdot\|_p \leq \|\cdot\|_q$. Intriguingly, this has also been fueled  by the development of local theory of Banach spaces (see \cite{HNVW, KST, LT}). 

The main point of this paper is to initiate the study of sharp constants in (forward) H\"older inequalities for sums of independent exponentials. In particular, for every $p \geq 2$, we find the largest constant $c_p$ such that
\begin{equation}\label{eq:Lp-L2-intro}
\|X\|_p \geq c_p\|X\|_2
 \end{equation}
holds for every random variable $X$ which is a sum of independent \emph{centred} exponential random variables. (This has been lately asked in \cite{MRTT}, see Sections 3.2 and 3.3 therein.)

We employ the standard notation that given $p \in \R$, for a random variable $X$, $\|X\|_p = (\E|X|^p)^{1/p}$ is its $L_p$ ``norm'' (which of course is \emph{not} a norm per se when $p < 1$).

\subsection{Algebraic motivation: Hunter's positivity theorem}
Even moments of sums of exponential random variables are also naturally linked to a fundamental algebraic object, the complete homogeneous symmetric (CHS) polynomials, which play a pivotal role in the classical theory of symmetric functions (see, e.g. \cite{Mac}). A classical result of Hunter from \cite{Hun} asserts that CHS polynomials are positive definite functions. In fact, he showed a quantitative lower bound which, probabilistically, amounts to \eqref{eq:Lp-L2-intro} with asymptotically sharp $c_p$, specialised to all even values of $p$. This has had interesting applications and connections to certain topics in analysis such as matricial norms \cite{ACGV}, positivity preservers \cite{KT}, or Maclaurin's inequalities \cite{Tao}. We also refer to the recent survey \cite{BCG}. The study of bounds on CHS polynomials has been very recently revived in \cite{BP}, where Hunter's theorem has been strengthen to a nonasymptotically sharp bound, among many other interesting results.

\subsection{Specifics and our results}

For the notation prevailing throughout this note, we let $\cE_1, \cE_2, \dots, \cE_1', \cE_2', \dots$, $\cE, \cE', \cE'', \dots$, etc. be independent identically distributed (i.i.d.) standard exponential random variables (that is, with density $e^{-x}\1_{[0,+\infty)}(x)$ on $\R$), with mean~$1$ and variance $1$.

As mentioned, the simplex-slicing directly pertains to this work. The main result of \cite{MRTT} establishes its probabilistic extension: for every log-concave random variable $X$ (that is, continuous with density of the form $e^{-V}$ for a convex function $V\colon \R \to (-\infty, +\infty])$), we have
\begin{equation}\label{eq:MRTT1}
\frac{\|X - \E X\|_p}{\|X - \E X\|_1}  \geq \|\cE - \cE'\|_p, \quad \text{for $-1 < p \leq 1$}
 \end{equation}
and
\begin{equation}\label{eq:MRTT2}
\frac{\|X - \E X\|_p}{\|X - \E X\|_1}  \leq \begin{cases} \|\cE - \cE'\|_p, & \text{for $1 \leq p \leq p_\star$} \\ \|\frac{2}{e}(\cE - 1)\|_p, & \text{for $p \geq p_\star$}, \end{cases}
 \end{equation}
where constant $p_\star = 2.9414..$ is defined as the unique solution to the equation $\|\cE - \cE'\|_p= \|\frac{2}{e}(\cE - 1)\|_p$ in $p \in (1, +\infty)$. We emphasise that $\|\cE - \cE'\|_1 = \|\frac{2}{e}(\cE - 1)\|_1 = 1$, so these bounds are sharp. Analysing appropriately the limiting case $p \downarrow -1$ of the first bound specialised to $X$ of the form $\sum x_j(\cE_j-1)$ allows to recover \cite{Webb} (for details, see \cite{MRTT}, Theorem 1 and the paragraph following it).

The complete homogeneous symmetric (CHS) polynomial $\hh_\ell(x)$ of degree $\ell$ in $n$ variables $x_1, \dots, x_n$ is the sum of all distinct degree-$\ell$ monomials in the given variables,
\begin{equation}\label{eq:def-h}
\hh_\ell(x) = \hh_\ell(x_1, \dots, x_n) = \sum_{1 \leq j_1 \leq \dots \leq j_\ell \leq n} x_{j_1}x_{j_2}\cdot\ldots\cdot x_{j_\ell}
\end{equation}
with the usual convention that $h_0(x) \equiv 1$.
Equivalently, by means of their generating functions, the CHS polynomials are defined by
\[ 
\prod_{j=1}^n \frac{1}{1-x_jt} = \sum_{\ell=0}^\infty \hh_\ell(x)t^\ell.
 \]
Since the left hand side is the moment generating function of the sum of independent exponentials with means $x_j$,
\[ 
\sum_{\ell=0}^\infty \frac{\E(\sum x_j\cE_j)^\ell}{\ell!}t^\ell = \E e^{t\sum_{j=1}^n x_j\cE_j} = \prod_{j=1}^n \E e^{tx_j\cE_j} = \prod_{j=1}^n \frac{1}{1-x_jt}, \qquad |t| < \min_j \big(|x_j|^{-1}\big),
 \]
the natural probabilistic representation of CHS polynomials as moments of sums of exponentials presents itself as follows,
\[ 
\hh_\ell(x) = \frac{1}{\ell!}\E\left(\sum_{j=1}^n x_j\cE_j\right)^\ell, \qquad \ell = 0, 1, \dots.
 \]
In view of this identity, it is obvious that $\hh_\ell(x) > 0$ for all even $\ell$ and all nonzero real vectors $x$ (note $\hh_\ell(-x) = (-1)^\ell\hh_\ell(x)$, so such positivity is bluntly false when $\ell$ is odd). On the other hand, based solely on the algebraic definition \eqref{eq:def-h}, this positivity property perhaps may not be so clear and Hunter's theorem of \cite{Hun} was conceived to explain that in a quantitative way, asserting the sharp bound,
\[
\hh_\ell(x) \geq \frac{1}{2^{\ell/2}(\ell/2)!}, \qquad \ell = 0, 2, 4, \dots, \qquad \text{provided that} \qquad \sum_{j=1}^n x_j^2 = 1.
 \]
We ought to mention in passing that with a modification of Hunter's argument, Tao has lately remarked a Schur-monotonicity result of the map $x \mapsto \hh_\ell(x)$ ($\ell$ even), which also gives Hunter's positivity theorem: $\hh_\ell(x) \geq 0$ for all even $\ell$, with strict inequality unless $x = 0$, see \cite{Tao-blog}.

Put equivalently as a forward H\"older-type inequality with a sharp constant, Hunter's theorem states
\begin{equation}\label{eq:Hunter}
\E\left|\sum_{j=1}^n x_j\cE_j\right|^\ell \geq (\E G^{\ell})\cdot \left(\sum_{j=1}^n x_j^2\right)^{\ell/2}, \qquad \ell = 0, 2, 4, \dots,
\end{equation}
where $G$ is a standard Gaussian random variable, so $\E G^\ell = (\ell-1)!! = \frac{\ell!}{2^{\ell/2}(\ell/2)!}$, $\ell = 0, 2, 4, \dots$.

Our main result extends this bound to all values of $\ell \geq 2$.

\begin{theorem}\label{thm:p>2}
Let $p \geq 2$. For every $n \geq 1$ and real numbers $x_1, \dots, x_n$, we have
\begin{equation}\label{eq:main-p>2}
\|X\|_p \geq \|G\|_p\sqrt{\Var(X)}, \qquad X = \sum_{j=1}^n x_j\cE_j,
\end{equation}
where $G \sim N(0,1)$ is a standard Gaussian random variable.
\end{theorem}

This provides a sharp reversal to \eqref{eq:MRTT2} for the $L_p$-variance comparison specialised to sums of independent exponentials, which of course becomes \eqref{eq:Lp-L2-intro} in the centred case, $\E X = 0$.

\begin{remark}\label{rem:conj-fixed-n}
This bound is (asymptotically) sharp as seen by taking $n = 2m$, $x = x_\star$ with $x_{\star, 1} = -x_{\star, 2} = x_{\star, 3} = -x_{\star, 4} = \dots = \frac{1}{\sqrt{2m}}$ and letting $m \to \infty$. We conjecture that  for a fixed \emph{even} number of summands $n$, the infimal value of the $L_p$ norm, $p \geq 2$, that is the infimal value of the left hand side of \eqref{eq:main-p>2} is attained at the balanced vector of half plus-minus ones, $x =  x_\star$. The very recent result of \cite{BP} strongly supports this, where the conjecture is established for all integral even values of exponent $p$.
\end{remark}

\begin{remark}\label{rem:relax}
A natural relaxation of \eqref{eq:main-p>2}, that is the minimisation problem of  the $L_p$-norm ($p \geq 2$) over a class of distributions from the class of sums of independent exponentials to the class of, say centred unimodal (or log-concave) random variables with the $L_2$-norm fixed leads to a different than Gaussian extremiser, namely the uniform one, see, e.g. Remark 16 in \cite{ENT2}. This is in striking contrast to \cite{MRTT}, where the original problem for the sums of exponentials is solved in the larger class of all log-concave centred random variables, leading to \eqref{eq:MRTT1} and \eqref{eq:MRTT2}. In that sense, bound \eqref{eq:main-p>2} somehow exploits the structure of sums of independent exponentials. 
\end{remark}

\begin{remark}\label{rem:neg-mom}
In a similar vein, where exponentials pivot an answer to an extremisation problem, we also ought to mention Tang's recent result \cite{T} on minimal volume sections of the simplex, saying that 
\[
f_{\sum x_j(\cE_j-1)} \geq f_{\cE-1}(0) = \frac{1}{e},
\]
in the setting of Theorem \ref{thm:p>2}, where $f_X$ here denotes the (continuous version of) density of a random variable $X$ (see also ``Fake Theorem 4'' therein). Equivalently (via a negative moment formula for the density, see, e.g. (1) in \cite{CNT}), 
\[
\lim_{p\downarrow -1} (1+p)\left\|\sum x_j(\cE_j-1)\right\|_p \leq \lim_{p\downarrow -1} (1+p)\left\|\cE-1\right\|_p.
 \]
In view of \eqref{eq:main-p>2}, we conjecture that for every $-1 < p < 2$,
\[ 
\sup\left\{\left\|\sum x_j(\cE_j-1)\right\|_p, \ n \geq 1, \sum_{j=1}^n x_j^2 = 1\right\} = \begin{cases} \|\cE-1\|_p, & -1 < p \leq p_0, \\ \|G\|_p, & p_0 \leq p < 2. \end{cases}
 \]
That is, there is a phase transition of the extremising distribution at $p_0 = -0.565..$, defined as a unique value of $p \in (-1, 2)$ for which the two $L_p$ norms on the right hand side coincide. 
\end{remark}

For low moments, we offer Schur-monotonicity results for sums with nonnegative coefficients (see, e.g. Chapter II in \cite{Bh} for very basics on majorisation). This setting allows to use the techniques of completely monotone functions (see \cite{Wid} for background).

Let $p > -1$ and define
\begin{equation}\label{eq:def-Mp}
M_p(x_1, \dots, x_n) = \E \left[\left(\sum_{j=1}^n \sqrt{x_j}\cE_j\right)^p\right], \qquad x_1, \dots, x_n \geq 0.
\end{equation}
We fully characterise the ranges of $p$ for which this function enjoys Schur-monotonicity.

\begin{theorem}\label{thm:Schur}
For an arbitrary integer $n \geq 2$, function $M_p$ is Schur-convex for $-1 < p < 0$, and Schur-concave for $0 \leq p \leq 4$ on $\R_+^n$. Moreover, this function fails to be Schur-monotone as long as $p > 4$.
\end{theorem}

The regime $-1 < p < 0$ has been addressed in \cite{CST0} (see Theorem 2 therein). Theorem \ref{thm:Schur} complements the recent results of \cite{BP} which provides sharp results for integral values of~$p$, as well as \cite{Tao-blog} which establishes the Schur-convexity under the usual $\ell_1$-constraint, viz. of the function $x \mapsto \hh_\ell(x)$.

The rest of this paper is devoted to our proofs of Theorems \ref{thm:p>2} and \ref{thm:Schur}, presented in Sections \ref{sec:proof-p>2} and \ref{sec:proof-Schur}, respectively. We finish off with concluding remarks

\section{Proofs: Hunter's positivity, beyond CHS (Theorem \ref{thm:p>2})}\label{sec:proof-p>2}

\subsection{Overview}
We build upon Hunter's original strategy from \cite{Hun}. At low-resolution, this entails a two-step approach. 

In the first step, by means of a local analysis of critical points one first reduces to the case $2 \leq p \leq 4$, leveraging certain algebraic identities enjoyed by exponential random variables. This is an inductive argument on $k = 1, 2, \dots$, where the veracity of \eqref{eq:main-p>2} for all $p \in [2k, 2k+2]$ is reduced to the base case $k=1$. This idea is also very much reminiscent of an inductive argument employed by Haagerup for Rademacher sums (see the final paragraph of \S 5 in \cite{Haa}), and, to some extent, is in the same spirit as an ``interpolation-lowering $p$'' trick developed and employed lately in Section 2 of \cite{BMNO}.

The second step deals with the range $2 \leq p \leq 4$, where we employ Fourier analytic formulae which have been widely explored in many similar contexts (notably, pioneered in Haagerup's work \cite{Haa}), but for \emph{upper bounds}. An interesting novel point of our argument is turning our foe, the lower bound on $L_p$--norm, into our ally, an upper bound on $L_{p-2}$--norm, the reduction made possible via a local analysis of extremisers, combined with the said algebraic identities, specific to the exponential distribution.

\subsection{Auxiliary lemmas}

Throughout this section, suppose $\Phi\colon \R\to \R$ is $C^1$ with $\Phi$ and $\Phi'$ growing moderately, say $|\Phi(t)|, |\Phi'(t)| = O(|t|^\beta)$ for some $\beta > 0$. Given a vector $x = (x_1, \dots, x_n)$ in $\R^n$, we form the sum
\[
S_x = \sum_{j=1}^n x_j\cE_j.
\]
We begin with a simple integration-by-parts-type formula for the exponential distribution.

\begin{lemma}\label{lm:by-parts}
For every real numbers $x_1, \dots, x_n$ and $u$, we have
\[ 
\E\Phi(S_x) = \E\Phi(S_x  + u\cE) - u\E\Phi'(S_x+u\cE).
 \]
\end{lemma}
\begin{proof}
Integration by parts indeed yields
\begin{align*}
\E\Phi(S_x + u\cE) &= \E\left[\int_0^\infty \Phi(S_x + ut)e^{-t} \dd t \right]\\
&= \E\left[-\Phi(S_x + ut)e^{-t}\Big|_{t=0}^\infty + \int_0^\infty u\Phi'(S_x + ut)e^{-t} \dd t\right] \\
&=\E\left[\Phi(S_x) + u\Phi'(S_x + u\cE)\right].
\end{align*}
Rearranging gives the result.
\end{proof}

We can now derive the aforementioned algebraic identities, quintessential for differentiating functionals of the form $\E \Phi(S_x)$, later used in the critical point analysis.

\begin{lemma}\label{lm:diff}
For every real numbers $x_1, \dots, x_n$ and $u, v$, $u \neq v$, we have
\[ 
\frac{\E\Phi(S_x + v\cE) - \E\Phi(S_x+u\cE)}{v-u} = \E\Phi'(S_x+u\cE + v\cE').
 \]
In particular,
\[ 
\frac{\partial}{\partial u} \E \Phi(S_x+u\cE) = \E\Phi'(S_x+u\cE + u\cE').
 \]
\end{lemma}
\begin{proof}
Applying Lemma \ref{lm:by-parts} to the sum $S_{(x, u)} = S_x + u\cE$ (with $n+1$ terms) gives,
\[ 
\E\Phi(S_x+u\cE) = \E\Phi(S_{(x,u)}) = \E\Phi(S_x + u\cE + v\cE') - v\E\Phi'(S_x + u\cE + v\cE').
 \]
Similarly,
\[ 
\E\Phi(S_x+v\cE) = \E\Phi(S_{(x,v)}) = \E\Phi(S_x + v\cE + u\cE') - u\E\Phi'(S_x + v\cE + u\cE').
 \]
Subtracting off from this the previous line yields
\[ 
\frac{\E\Phi(S_x + v\cE) - \E\Phi(S_x+u\cE)}{v-u} = \E\Phi'(S_x+u\cE + v\cE'),
 \]
as desired. Letting $v \to u$ readily gives
\[ 
\frac{\partial}{\partial u} \E \Phi(S_x+u\cE) = \E\Phi'(S_x+u\cE + u\cE').\qedhere
 \]
\end{proof}

\begin{remark}\label{rem:lm-diff-gen}
Note that in the proofs of both Lemmas \ref{lm:by-parts} and \ref{lm:diff}, we have only used the fact the $\cE$ and $\cE'$ have the exponential distribution, and that the $S_x$ is an independent of them random variable, of in fact \emph{arbitrary} distribution (provided the integrability). 
\end{remark}

The next lemma evaluates an elementary integral arising in certain bounds in the Fourier analytical part of the main argument.

\begin{lemma}\label{lm:integral} For $s>0$ and $0<q<2$, we have
\[
\int_{0}^{\infty}\frac{1-\bigl(1+t^{2}/s\bigr)^{-\frac{1+s}{2}}}{t^{q+1}}\,\dd t =\frac{1}{q}\,\Gamma\!\Bigl(1-\frac q2\Bigr)\,
		\frac{\Gamma\!\Bigl(\frac{1+q+s}{2}\Bigr)}{s^{q/2}\,\Gamma\!\Bigl(\frac{1+s}{2}\Bigr)}.
\]
\end{lemma}

\begin{proof}
Denote the integral on the left hand side by $I_{q,s}$.
Set $\alpha=\frac{1+s}{2}$ and $\beta=\frac q2$.
With the change of variables $u=t^{2}/s$, we have 
\[
		I_{q,s}
		=\frac12 s^{-q/2}\int_{0}^{\infty}\bigl[1-(1+u)^{-\alpha}\bigr]u^{-1-\beta}\,\dd u.
\]
Using the Gamma representation
\[
		(1+u)^{-\alpha}=\frac{1}{\Gamma(\alpha)}\int_{0}^{\infty}x^{\alpha-1}e^{-x}e^{-ux}\,\dd x,
\]
we obtain
\[
		1-(1+u)^{-\alpha}
		=\frac{1}{\Gamma(\alpha)}\int_{0}^{\infty}x^{\alpha-1}e^{-x}\bigl(1-e^{-ux}\bigr)\,\dd x.
\]
Plugging this back and exchanging the order of integration,
\[
		\int_{0}^{\infty}\bigl[1-(1+u)^{-\alpha}\bigr]u^{-1-\beta}\,\dd u
		=\frac{1}{\Gamma(\alpha)}\int_{0}^{\infty}x^{\alpha-1}e^{-x}
		\left(\int_{0}^{\infty}(1-e^{-ux})u^{-1-\beta}\,du\right)\,\dd x.
\]
For $0<\beta<1$, changing the variables and integrating by parts,
\[
		\int_{0}^{\infty}(1-e^{-ux})u^{-1-\beta}\,\dd u = x^\beta\int_0^\infty (1-e^{-u})u^{-1-\beta}\dd u 
		=x^{\beta}\frac{\Gamma(1-\beta)}{\beta}
\]
hence
\[
		\int_{0}^{\infty}\bigl[1-(1+u)^{-\alpha}\bigr]u^{-1-\beta}\,du
		=\frac{\Gamma(1-\beta)}{\beta}\cdot \frac{1}{\Gamma(\alpha)}
		\int_{0}^{\infty}x^{\alpha+\beta-1}e^{-x}\,dx
		=\frac{\Gamma(1-\beta)}{\beta}\cdot\frac{\Gamma(\alpha+\beta)}{\Gamma(\alpha)}.
\]
Therefore,
\[
		I_{q,s}
		=\frac12 s^{-q/2}\cdot \frac{\Gamma(1-\beta)}{\beta}\cdot\frac{\Gamma(\alpha+\beta)}{\Gamma(\alpha)}
		=\frac{1}{q}\,\Gamma\!\Bigl(1-\frac q2\Bigr)\,
		\frac{\Gamma\!\Bigl(\frac{1+q+s}{2}\Bigr)}{s^{q/2}\,\Gamma\!\Bigl(\frac{1+s}{2}\Bigr)}.\qedhere
\]
\end{proof}

We shall need a technical result on the special function from the previous lemma. In fact, this is the backbone of our future bounds. Its proof is inspired by Haagerup's Lemma 1.4 in \cite{Haa}.

\begin{lemma}\label{lm:special-fun-monot}
For $\beta, x > 0$, let
\[ 
\Psi_\beta(x) = \frac{\Gamma\left(x + \beta + \frac12\right)}{x^\beta\Gamma\left(x+\frac12\right)}.
 \]
Then, for every $\beta > 0$, 
\[
(0, +\infty) \ni x \mapsto \Psi_\beta(x) \text{ is strictly decreasing and }\Psi_\beta(x) \to 1 \text{ as } x \to \infty.
\]
\end{lemma}
\begin{proof}
Using Stirling's formula, $\Gamma(x) \sim \sqrt{2\pi}x^{x-1/2}e^{-x}$, as $x \to \infty$, we have
\[ 
\Psi_\beta(x) \sim \frac{\left(x+\beta+\frac12\right)^{x+\beta}}{x^\beta\left(x+\frac12\right)^x}e^{-\beta} = \left(1 + \frac{\beta+\frac12}{x}\right)^\beta\left(1 + \frac{\beta}{x+\frac12}\right)^xe^{-\beta} \sim 1.
 \]
Fix $x > 0$. Note that
\[ 
\Psi_\beta(x+1) = \frac{\left(x+\beta+\frac12\right)\Gamma\left(x + \beta + \frac12\right)}{(x+1)^\beta\left(x+\frac12\right)\Gamma\left(x+\frac12\right)} = \left(\frac{x}{x+1}\right)^\beta\frac{x+\beta+\frac12}{x+\frac12}\cdot \Psi_\beta(x).
 \]
Rewriting with
\[ 
R_\beta(x) = \left(1+\frac1x\right)^\beta\frac{x+\frac12}{x+\beta+\frac12},
 \]
and iterating  $n$-times yields
\begin{align*}
\Psi_\beta(x) &= \Psi_\beta(x+1)R_\beta(x) = \Psi_\beta(x+2)R_\beta(x)R_\beta(x+1) \\
&= \dots =\Psi_\beta(x+n+1) \prod_{k=0}^n R_\beta(x+k).
 \end{align*}
Thus letting $n \to \infty$, we find that
\[
\Psi_\beta(x) = \prod_{k=0}^\infty R_\beta(x+k).
\]
The following claim finishes the argument.

\textbf{Claim.} For every $\beta > 0$, the function $x \mapsto R_\beta(x)$ is strictly decreasing on $(0,+\infty)$.

Indeed,
\begin{align*}
\frac{\dd}{\dd x}\log R_\beta(x) &= \beta\frac{1}{1+1/x}\frac{-1}{x^2} + \frac{1}{x+\frac12} - \frac{1}{x+\beta + \frac12} \\
&= -\frac{\beta}{x(x+1)} + \frac{\beta}{\left(x+\frac12\right)\left(x+\beta+\frac12\right)} \\
&< -\frac{\beta}{x(x+1)} + \frac{\beta}{\left(x+\frac12\right)^2} < 0.\qquad\qquad\qedhere
 \end{align*}
\end{proof}


Finally, we need to check directly that Theorem \ref{thm:p>2} holds in the special case when all coefficients $x_j$ are equal. In fact, there is a stronger bound.

\begin{lemma}\label{lm:all-equal}
For every $n = 1, 2, \dots$, and $p \geq 2$, we have
\[ 
\E\left|\frac{\cE_1 + \dots + \cE_n}{\sqrt{n}}\right|^p \geq 2^{p/2}\E|G|^p, \qquad G \sim N(0,1).
 \]
\end{lemma}
\begin{proof}
It is an elementary fact that the symmetric exponential distribution is a Gaussian mixture. Specifically, we have the following identity for distributions,
\[ 
\cE - \cE' \ \overset{d}{=} \sqrt{2\cE}G,
 \]
where $G$ is a standard Gaussian random variable independent of the exponential random variable $\cE$. This identity can perhaps be most conveniently checked by comparing the characteristic functions, see also Lemma 23 in \cite{ENT1} and Remark (i) following it. Therefore,
\[ 
\cE_j \overset{d}{=} |\cE_j - \cE_j'| \overset{d}{=}  \sqrt{2\cE_j}|G_j|,
 \]
where $G_1, G_2, \dots$ are i.i.d. standard Gaussian random variables, independent of the $\cE_j$. Thus,
\[ 
\E\left|\sum_{j=1}^n \cE_j\right|^p = 2^{p/2}\E\left|\sum_{j=1}^n \sqrt{\cE_j}|G_j|\right|^p.
 \]
Note that a crude application of the triangle inequality gives: $\sum \sqrt{\cE_j}|G_j| \geq \Big|\sum \sqrt{\cE_j}G_j\Big|$, a.s. Using independence, the sum $\sum \sqrt{\cE_j}G_j$ has the same distribution as $\sqrt{\sum \cE_j}G_1$. As a result,
\[ 
\E\left|\sum_{j=1}^n \cE_j\right|^p \geq 2^{p/2}\E\Bigg(\sum_{j=1}^n \cE_j\Bigg)^{p/2}\cdot \E|G|^p.
 \]
For $p \geq 2$, by convexity, 
\[
\E\Bigg(\sum_{j=1}^n \cE_j\Bigg)^{p/2} \geq \Bigg(\E\sum_{j=1}^n \cE_j\Bigg)^{p/2} = n^{p/2}.\qedhere
\]
\end{proof}

\subsection{Proof of Theorem \ref{thm:p>2}}

The whole proof runs by  induction on $k = 1, 2, \dots$ for the statement that \eqref{eq:main-p>2} holds for \emph{all} $p \in [2k, 2k+2]$ with arbitrary sum $X = \sum_{j=1}^n x_j\cE_j$, $n \geq 1$. The base case $k = 1$ will be handled in Step II. 

\emph{Step I: A reduction to $2 \leq p \leq 4$ via Hunter's local argument.} 
For the inductive step: suppose \eqref{eq:main-p>2} holds for all $p \in [2k-2, 2k]$ for some $k \geq 2$ and fix $p \in [2k, 2k+2]$, $n \geq 1$. We consider critical points of
\[ 
(x_1, \dots, x_n) \mapsto \E\Phi(S_x), \qquad \text{subject to} \qquad \sum_{j=1}^n x_j^2 = 1,
 \]
where $\Phi(x) = |x|^p$. Let $x$ be such a critical point, and to finish the inductive step, it suffices to show that  \eqref{eq:main-p>2} holds with $X = S_x$. Using criticality, for each $1 \leq j \leq n$, we have
\[ 
\frac{\partial}{\partial x_j}\E\Phi(S_x) +2\lambda x_j = 0
 \]
for some Lagrange's multiplier $\lambda \in \R$. Multiplying by $x_j$ and using the $p$-homogeneity of $\Phi$ yields
\[ 
0 = \sum_{j=1}^n x_j\frac{\partial}{\partial x_j}\E\Phi(S_x) + 2\lambda\sum_{j=1}^n x_j^2 = p\E\Phi(S_x) + 2\lambda,
 \]
that is, at a critical point $x$, we have
\[ 
-2\lambda = p\E\Phi(S_x).
 \]
On the other hand, by Lemma \ref{lm:diff},
\[ 
\frac{\partial}{\partial x_j}\E\Phi(S_x) = \E \Phi'(S_x+x_j\cE),
 \]
which combined with Lagrange's equations results in
\[ 
x_jp\E\Phi(S_x) = x_j(-2\lambda) = \frac{\partial}{\partial x_j}\E\Phi(S_x) = \E\Phi'(S_x + x_j\cE).
 \]

If all $x_j$ are equal, we are done by Lemma \ref{lm:all-equal}. Otherwise, say $x_1 \neq x_2$, and we continue with the inductive argument. Using the above identity,
\[
p\E\Phi(S_x) = p\frac{x_1E\Phi(S_x) - x_2\E\Phi(S_x)}{x_1-x_2} 
= \frac{\E\Phi'(S_x + x_1\cE) - \E\Phi'(S_x + x_2\cE)}{x_1-x_2} 
\]
and invoking  Lemma \ref{lm:diff} one more time, we arrive at the crux, which is the identity for the critical point $x$,
\begin{equation}\label{eq:crux}
p\E\Phi(S_x)  = \E\Phi''(S_x + x_1\cE + x_2\cE').
\end{equation}
Plainly, $\Phi'(x) = p|x|^{p-1}\sgn(x)$, $\Phi''(x) = p(p-1)|x|^{p-2}$, so the inductive hypothesis gives
\begin{align*}
\E\Phi''(S_x + x_1\cE + x_2\cE') &\geq p(p-1)\E|G|^{p-2}\cdot (1+x_1^2+x_2^2)^{\frac{p-2}{2}} \\
&> p(p-1)\E|G|^{p-2}.
 \end{align*}
Finally, a straightforward calculation using integration by parts yields
\[ 
(p-1)\E|G|^{p-2} = \E|G|^p,
 \]
thus
\[ 
\E\Phi(S_x) > \E|G|^p.
 \]

\emph{Step II: Proof of \eqref{eq:main-p>2} for $2 \leq p \leq 4$:}
We again dispose of the cases when all $x_j$ are equal by means of Lemma \ref{lm:all-equal}. 

As derived in Step I, at every critical point $x$ of the map $x \mapsto \E|S_x|^p$ subject to $\sum x_j^2 = 1$, we have \eqref{eq:crux}, that is,
\[ 
\E|S_x|^p = (p-1)\E|S_x + x_1\cE_1' + x_2\cE_2'|^{p-2},
 \]
provided that $x_1 \neq x_2$. Without loss of generality, $x_1 = \|x\|_\infty$ (relabelling the components of $x$ and flipping the sign of $x_1$ if needed). Denote 
\[
Y =  S_x + x_1\cE_1' + x_2\cE_2'
\]
and $q = p-2$. We can assume that $0 < q < 2$. Using the standard Fourier-analytic formula (see, e.g. (2f) in \cite{Haa}, or Lemma 4 in \cite{CGT}),
\begin{equation}\label{eq:Fourier}
\E|Y|^q = c_q\int_0^\infty \frac{1 - \mathrm{Re}(\phi_Y(t))}{t^{q+1}} \dd t, \qquad c_q = \frac{2}{\pi}\sin\left(\frac{\pi q}{2}\right)\Gamma(q+1),
 \end{equation}
where $\phi_Y(t) = \E e^{it Y}$ is the characteristic function of the random variable $Y$. Here, thanks to independence,
\[ 
\E e^{it Y} =\E e^{itx_1\cE}\E e^{itx_2 \cE}  \E e^{it S_x}= \frac{1}{1-itx_1}\frac{1}{1-itx_2}\prod_{j=1}^n \frac{1}{1-itx_j}.
 \]
Note that
\[ 
 \mathrm{Re}(\phi_Y(t)) \leq |\phi_Y(t)| = (1+x_1^2t^2)^{-1}(1+x_2^2t^2)^{-1}\prod_{j=3}^n (1+x_j^2t^2)^{-1/2}.
 \]
The function $u \mapsto (1+ut^2)^{1/u}$ is decreasing in $u$ on $(0,+\infty)$ (for every fixed $t$). As a result, $(1+bt^2) \geq (1+at^2)^{b/a}$ for every $0 < b \leq a$, hence
\begin{align*}
\mathrm{Re}(\phi_Y(t)) &\leq (1+x_1^2t^2)^{-1}(1+x_1^2t^2)^{-x_2^2/x_1^2}\prod_{j=3}^n (1+x_1^2t^2)^{-x_j^2/(2x_1^2)} \\
&= (1+x_1^2t^2)^{-\frac{x_1^2+x_2^2+1}{2x_1^2}} \\
&\leq (1+x_1^2t^2)^{-\frac{x_1^2+1}{2x_1^2}}.
 \end{align*}
Let $s = x_1^{-2}$. Plugging this back yields the bound
\begin{align*}
\E|Y|^q &\geq c_q\int_0^\infty \frac{1 - (1+t^2/s)^{-\frac{1+s}{2}}}{t^{q+1}} \dd t \\
&= c_q\frac{1}{q}\Gamma\left(1-\frac{q}{2}\right)\frac{\Gamma\left(\frac{1+q+s}{2}\right)}{s^{q/2}\Gamma\left(\frac{1+s}{2}\right)} = c_q\frac{1}{q}\Gamma\left(1-\frac{q}{2}\right)2^{-q/2}\Psi_{q/2}\left(\frac{s}{2}\right),
 \end{align*}
where the first equality is an elementary calculation involving the Gamma function derived in Lemma \ref{lm:integral}, rewritten in the second equality in terms of special function $\Psi_\cdot(\cdot)$ from Lemma \ref{lm:special-fun-monot}. Moreover, by Lemma \ref{lm:special-fun-monot}, we obtain the lower bound on the right hand side by its limit as $s \to \infty$ resulting with
\[ 
\E|Y|^q \geq c_q\lim_{s\to\infty}\int_0^\infty \frac{1 - (1+t^2/s)^{-\frac{1+s}{2}}}{t^{q+1}} \dd t = c_q\int_0^\infty \frac{1 - e^{-t^2/2}}{t^{q+1}} \dd t = \E|G|^q.
 \]
Consequently,
\[ 
\E|S_x|^p \geq (p-1)\E|G|^{p-2} = \E|G|^p,
 \]
which finishes the proof. \qed

\section{Proofs: Schur-monotonicity of low moments (Theorem \ref{thm:Schur})}\label{sec:proof-Schur}

\subsection{Overview}
The analytic underpinning of our approach comprises integral representations for power functions $x \mapsto x^p$ in terms of mixtures of exponential ones. This naturally exploits the simple fact that a sufficiently high derivative of $x^p$ becomes a power function with a negative exponent, which is completely monotone, and thus enjoys the well-known elementary identity,
\[ 
x^{-q} = \frac{1}{\Gamma(q)}\int_0^\infty t^{q-1}e^{-tx} \dd t, \qquad x, q > 0,
 \]
akin to \eqref{eq:Fourier}, allowing to leverage independence. Similar ideas have been recently implemented for instance in \cite{CST0, LamT}. The obvious restriction of this method is that it only allows to handle nonnegative random variables.

\subsection{Integral representations}

We begin with a point-wise integral formula for power functions.

\begin{lemma}\label{lm:Q_k}
Let $k$ be a nonnegative integer and let $k < p < k+1$. Define
\[
Q_k(t) = (-1)^{k+1}\left(e^{-t} - \sum_{j=0}^k \frac{(-1)^j}{j!}t^j\right), \qquad t > 0.
\]
Then $Q_k(t) > 0$, the integral
\[
C_p = \int_0^\infty Q_k(t)t^{-p-1}\dd t
\]
converges and for $x > 0$, we have
\[
x^p = C_p^{-1}\int_0^\infty Q_k(tx)t^{-p-1}\dd t.
\]
\end{lemma}
\begin{proof}
For $t > 0$, by Taylor's formula with Lagrange's remainder,
\[
Q_k(t) = \frac{e^{-\theta}t^{k+1}}{(k+1)!}
\]
for some $0 < \theta < t$, so in particular the left hand side is positive and $O(t^{k+1})$ as $t \to 0$. Therefore, $C_p > 0$ is well-defined, i.e. the integral defining $C_p$ converges (it converges if and only if $k+1 - p - 1 > -1$ and $-p-1+k < -1$, that is $k < p < k+1$). For a fixed $x > 0$, the change of variables $t \mapsto tx$ in the integral defining $C_p$ yields the formula for $x^p$.
\end{proof}

For a nonnegative integer $k$, we let
\[
F_k(x_1, \dots, x_n) = \E Q_k\left(\sum_{j=1}^n \sqrt{x_j}\cE_j\right),
\]
where $Q_k$ is defined in Lemma \ref{lm:Q_k}.
Averaging over the distribution of sums of independent exponentials leads to the desired integral representations for $M_p$, recall \eqref{eq:def-Mp}. 

\begin{lemma}\label{lm:Mp-formula}
Let $k$ be a nonnegative integer and let $k < p < k+1$. Then
\[ 
M_p(x_1, \dots, x_n) = C_p^{-1}\int_0^\infty F_k(t^2x_1, \dots, t^2x_n)t^{-p-1}\dd t.
 \]
\end{lemma}

\subsection{Proof of Theorem \ref{thm:Schur}}

In view of Lemma \ref{lm:Mp-formula}, it suffices to prove the following results.

\begin{lemma}\label{lm:F}
Functions $F_\ell$ are Schur-convex for $\ell = 0, 1, 2, 3$.
\end{lemma}

\begin{lemma}\label{lm:failure}
Let $p > 4$. Function $x \mapsto M_p(x^2,1-x^2)$ has a local maximum in $(0, \frac{1}{\sqrt2})$, hence it fails to be monotone on $(0, \frac{1}{\sqrt2})$, hence $M_p$ is neither Schur-convex, nor Schur-concave.
\end{lemma}

\begin{proof}[Proof of Lemma \ref{lm:F}]
We have, $\E \cE_j^m = m!$, thus
\begin{align*}
P(x) &= \E \exp\left(-\sum \sqrt{x_j}\cE_j\right) = \prod \frac{1}{1+\sqrt{x_j}}, \\
M_1(x) &= \E \left(\sum \sqrt{x_j}\cE_j\right) = \sum \sqrt{x_j}, \\
M_2(x) &= \E \left(\sum \sqrt{x_j}\cE_j\right)^2 = 2\sum x_j + \sum_{j \neq k} \sqrt{x_jx_k}, \\
M_3(x) &= \E \left(\sum \sqrt{x_j}\cE_j\right)^3 = 6\sum x_j^{3/2} + 6\sum_{j \neq k} x_j\sqrt{x_k} + \sum_{i \neq j \neq k} \sqrt{x_ix_jx_k}.
\end{align*}
We shall use the Ostrowski criterion, so we will also need the partial derivatives,
\begin{align*}
\frac{\partial P}{\partial x_1} &= \frac{-\frac{1}{2\sqrt{x_1}}}{(1+\sqrt{x_1})^2}\prod_{j\neq 1} \frac{1}{1+\sqrt{x_j}} = -\frac{1}{2\sqrt{x_1}(1+\sqrt{x_1})}P(x),\\
\frac{\partial M_1}{\partial x_1} &= \frac{1}{2\sqrt{x_1}},\\
\frac{\partial M_2}{\partial x_1} &= 2 + \frac{1}{\sqrt{x_1}}\sum_{j\neq 1}\sqrt{x_j} = 1+\frac{M_1(x)}{\sqrt{x_1}},\\
\frac{\partial M_3}{\partial x_1} &= 9\sqrt{x_1} + 6\sum_{j\neq 1}\sqrt{x_j} + \frac{3}{\sqrt{x_1}}\sum_{j\neq 1}x_j + \frac{3}{2\sqrt{x_1}}\sum_{j\neq k \neq 1} \sqrt{x_jx_k} \\
&= 3M_1(x) + \frac{3}{\sqrt{x_1}}\sum x_j + \frac{3}{2\sqrt{x_1}}\sum_{j\neq k \neq 1} \sqrt{x_jx_k}\\
&=3\sqrt{x_1}+\frac{3}{\sqrt{x_1}}\sum x_j + \frac{3}{2\sqrt{x_1}}\sum_{j\neq k} \sqrt{x_jx_k}
\end{align*}
For convenience, we will also use the variables $b_j \equiv \sqrt{x_j}$.

\emph{Case $\ell = 0$.} We have, $Q_0(t) = 1 - e^{-t}$, thus
\[
F_0(x) = 1 - P(x),
\]
therefore,
\[
\frac{\partial F_0}{\partial x_1} - \frac{\partial F_0}{\partial x_2} = \frac12\left(\frac{1}{b_1(1+b_1)} - \frac{1}{b_2(1+b_2)}\right)P(x) = \frac{(b_2-b_1)(1+b_1+b_2)}{2b_1b_2(1+b_1)(1+b_2)}P(x).
\]
When $x_1 > x_2$, equivalently $b_1 > b_2$, this difference of partial derivatives is negative, so $F_0$ is Schur-concave.

\emph{Case $\ell = 1$.}
We have, $Q_1(t) = e^{-t} -1 + t$, thus
\[
F_1(x) = P(x)-1 + M_1(x),
\]
therefore,
\begin{align*}
\frac{\partial F_1}{\partial x_1} - \frac{\partial F_1}{\partial x_2} &= -\frac{(b_2-b_1)(1+b_1+b_2)}{2b_1b_2(1+b_1)(1+b_2)}P(x) + \frac{b_2-b_1}{2b_1b_2} \\
&=\left(\frac{b_2-b_1}{2b_1b_2}\right)\left(1 - \frac{1+b_1+b_2}{(1+b_1)(1+b_2)}P(x)\right).
\end{align*}
Plainly $\frac{1+b_1+b_2}{(1+b_1)(1+b_2)} < 1$ and $P(x) < 1$, so the second bracket is positive.
As a result, when $x_1 > x_2$, equivalently $b_1 > b_2$, this difference of partial derivatives is negative, so $F_1$ is Schur-concave.

\emph{Case $\ell = 2$.}
We have, $Q_2(t) = -e^{-t} +1 - t + \frac{t^2}{2}$, thus
\[
F_2(x) = -P(x) + 1 - M_1(x) + \frac{1}{2}M_2(x),
\]
therefore,
\begin{align*}
\frac{\partial F_2}{\partial x_1} - \frac{\partial F_2}{\partial x_2} &= \frac{(b_2-b_1)(1+b_1+b_2)}{2b_1b_2(1+b_1)(1+b_2)}P(x) - \frac{b_2-b_1}{2b_1b_2} + \frac{b_2-b_1}{2b_1b_2}M_1(x)  \\
&=\left(\frac{b_2-b_1}{2b_1b_2}\right)\left(\frac{1+b_1+b_2}{(1+b_1)(1+b_2)}P(x)-1 + M_1(x)\right).
\end{align*}
This is perhaps less obvious, but the second bracket is positive.

\textbf{Claim.} For positive numbers $b_1, \dots, b_n$, we have
\[
\frac{1+b_1+b_2}{(1+b_1)(1+b_2)} > \left(1 - \sum b_j\right)\prod (1+b_j).
\]
As a result, when $x_1 > x_2$, equivalently $b_1 > b_2$, this difference of partial derivatives is negative, so $F_2$ is Schur-concave.

\emph{Case $\ell = 3$.} Again, we calculate that
\begin{align*}
	&\frac{\partial F_3}{\partial x_1} - \frac{\partial F_3}{\partial x_2}=\\ &\left(\frac{b_2-b_1}{2b_1b_2}\right)\left(6\sum b_i^2+3\sum_{i\neq j}b_ib_j-6b_1b_2-\frac{1+b_1+b_2}{(1+b_1)(1+b_2)}P(x)+1-M_1(x)\right).
\end{align*}
Note that $$6\sum b_i^2+3\sum_{i\neq j}b_ib_j-6b_1b_2-M_1(x)=3+3\left(\sum b_i\right)^2-6b_1b_2-\sum b_i\geq 0,$$
since $$6b_1b_2\leq\frac{3}{2}(b_1+b_2)^2\leq \frac{3}{2}(\sum b_i)^2$$ and 
$$\frac{1}{2}(\sum b_i)^2+\frac{1}{2}\geq\sum b_i.$$
\end{proof}

\begin{proof}[Proof of the claim.]
When $\sum b_j \geq 1$, the inequality is obvious, so suppose $\sum b_j  <1$. Note that
\[
\left(1 - \sum b_j\right)\prod (1+b_j) \leq e^{-\sum b_j}(1+b_1)(1+b_2)e^{\sum_{j>2} b_j} = e^{-(b_1+b_2)}(1+b_1)(1+b_2).
\]
Let $s = b_1 + b_2$. It suffices to show that
\[
(1+s)e^s > (1+b_1)^2(1+b_2)^2 = (1+s + b_1b_2)^2
\]
Since $b_1b_2 \leq (s/2)^2$, it suffices to show that
\[
(1+s)e^s > (1+s/2)^4, \qquad 0 < s  <1.
\]
This follows from
\[
(1+s)e^s - (1+s/2)^4 > (1+s)(1+s+s^2/2+s^3/6) - (1+s/2)^4 = s^3/6 + 5s^4/48 > 0.
\]
\end{proof}

\begin{proof}[Proof of Lemma \ref{lm:failure}]
For $a, b > 0$, the density of $a\cE + b\cE'$ at $x \geq 0$ is $\frac{e^{-x/a}-e^{-x/b}}{a-b}$, thus
\[ 
M_p(a^2, b^2) = \E(a\cE + b\cE')^p = \Gamma(p+1)\frac{b^{p+1}-a^{p+1}}{b-a}.
 \]
We consider
\[ 
f(x) = \frac{1}{\Gamma(p+1)}M_p(x^2, 1-x^2) = \frac{(1-x^2)^{\frac{p+1}{2}} - x^{p+1}}{\sqrt{1-x^2}-x} , \qquad 0 \leq x \leq \frac{1}{\sqrt{2}}
 \]
and argue that for a fixed $p > 4$, $f(x)$ fails to be monotone. We check that
\[ 
f'(0) = 1, \qquad f'(1/\sqrt{2}) = 0, \qquad f''(1/\sqrt{2}) = \frac{1}{3}2^{1-p/2}p(p+1)(p-4) > 0.
 \]
That means $f'(x)$ is strictly increasing near $x = \frac{1}{\sqrt{2}}$, thus $f'(x_0) < 0$ for some $x_0 < \frac{1}{\sqrt{2}}$. Consequently, $f'$ vanishes at some point in $(0, x_0)$, where it attains a local maximum.
\end{proof}

\section{Concluding remarks}

We would like to finish this paper with suggesting two rather natural avenues to strengthen our results, left for future investigations.

\subsection{The gamma distribution}

Theorem \ref{thm:p>2} can be generalised to weighted sums of independent Gamma random variables, in that \eqref{eq:main-p>2} still holds whenever $\cE_j$ has the $\Gamma(\gamma_j)$ distribution, with arbitrary $\gamma_j > 0$, $j = 1, \dots, n$ (i.e. with density $\Gamma(\gamma_j)^{-1}x^{\gamma_j-1}e^{-x}\1_{[0,+\infty)}(x)$). 

Our proof of Theorem \ref{thm:p>2} can be repeated almost verbatim for this setting. Of course one needs to first generalise the algebraic identities from Lemmas \ref{lm:by-parts} and \ref{lm:diff}: when one goes over the crucial calculation for the Lagrange's equations leading to \eqref{eq:crux}, it turns out that one only needs to supplement the said lemmas with Remark \ref{rem:lm-diff-gen} and the following identify: for a gamma random variable $X \sim \Gamma(\gamma)$ and independent exponential $\cE$, we have
\[ 
\E\Big[X \Phi(X)\Big ] = \gamma\E\Phi(X+\cE).
 \]
Another main point is that the characteristic function of the gamma distribution has \emph{almost} the same form as for the exponential one, $\phi_X(t) = (1-it)^{-\gamma}$, if $X \sim \Gamma(\gamma)$. Consequently, the crucial technical calculations and estimates from Lemmas \ref{lm:integral} and \ref{lm:special-fun-monot} can be employed unchanged. We omit the further details. 

For better clarity and transparency of main ideas, we have decided to keep the exposition of Theorem \ref{thm:p>2} specialised to the exponential distribution. Additionally, we feel that our proof scheme is robust enough to handle even more general settings, and it will be of interest to investigate that elsewhere.

The same goes for Theorem \ref{thm:Schur} --- in light of the fact that the result from \cite{CST0} pertaining to the case $-1 < p < 0$ is in fact proved therein for sums of independent Gamma random variables, we believe Theorem \ref{thm:Schur} continues to hold in this setting as well. Moreover, it would perhaps be meaningful to extend \cite{CST0} as well as our Theorem \ref{thm:Schur} beyond that setting.

\subsection{Log-convexity}

We conjecture that the Gaussian extremiser featured in Hunter's inequality \eqref{eq:Hunter} and in our Theorem \ref{thm:p>2}, \eqref{eq:main-p>2}, persists in the more general situation of comparing \emph{arbitrary} $L_p$ and $L_q$ norms of weighted sums of independent exponential random variables with mean $0$, as long as $p, q  \geq 2$. Namely, given $2 \leq p \leq q$, for every $n \geq 1$ and real numbers $x_1, \dots, x_n$ with $\sum_{j=1}^n x_j = 0$, we have
\[ 
\|X\|_p \leq \frac{\|G\|_p}{\|G\|_q}\|X\|_q, \qquad X = \sum_{j=1}^n x_j\cE_j.
 \]
The stipulation that $p \geq 2$ perhaps is \emph{not} stringent, but in light of the discussion from Remark \ref{rem:neg-mom},  for such inequality to hold, $p$ needs to be bounded away from $-1$. The stipulation that $\sum x_j = 0$ seems natural and, at the same time, some conditions forbidding $X$ from being close to shifted exponential random variables $\cE + a$ seem necessary --- it can be checked that the inequality fails as $a \to +\infty$. Finally, the stipulation that $X$ is a a sum of exponentials also is important and makes the question rather intriguing, as the obvious relaxation to the class of mean $0$ log-concave random variables cannot feature the Gaussian extremiser (by localisation methods, the extremiser is within the class of random variables with the piece-wise log-affine densities, see, e.g. discussions in Section 3.2 in \cite{MRTT}).

There is a compelling line of attack, via a slightly more general statement that given arbitrary real coefficients $x_1, \dots, x_n$ with $\sum_{j=1}^n x_j = 0$, the function
\[ 
p \mapsto \frac{\E|X|^p}{\E|G|^p}, \qquad X = \sum_{j=1}^n x_j\cE_j
 \]
is log-convex on $(2, +\infty)$. This is supported by the special \emph{symmetric} case: when  $n$ is even and $x_1 = -x_2$, $x_3 = -x_4, \ldots$, we can use the Gaussian mixture structure of the symmetric exponential distribution, as in the proof of Lemma \ref{lm:all-equal}, to conclude that the distribution of $X$ can be represented in the form $R\cdot G$, for a nonnegative random variable $R$, independent of $G$. Then, plainly,
$\frac{\E|X|^p}{\E|G|^p} = \E R^p$
is log-convex.




\begin{thebibliography}{9}


\bibitem{ACGV}
Aguilar, K., Chávez, Á., Garcia, S. R., Volčič, J., Norms on complex matrices induced by complete homogeneous symmetric polynomials. 
Bull. Lond. Math. Soc. 54 (2022), no. 6, 2078--2100.

\bibitem{Bal}
Balakrishnan, K., Exponential Distribution: Theory, Methods and Applications (1st ed.). Routledge, London, 1999.

\bibitem{Ball-cube}
Ball, K.,
Cube slicing in $\R^n$. 
Proc. Amer. Math. Soc. 97 (1986), no. 3, 465--473.



\bibitem{BMNO}
Bara\'nski, A., Murawski, D., Nayar, P., Oleszkiewicz, K.,
On the optimal $L_p-L_4$ Khintchine inequality.
Preprint (2025): arXiv:2503.11869.


\bibitem{BN}
Barthe, F., Naor, A.,
Hyperplane projections of the unit ball of $\ell_p^n$. 
Discrete Comput. Geom. 27 (2002), no. 2, 215--226. 



\bibitem{Bh}
Bhatia, R.,
Matrix analysis. 
Graduate Texts in Mathematics, 169. Springer-Verlag, New York, 1997.


\bibitem{BCG}
Bouthat, L., Chávez, Á., Garcia, S. R.,
Hunter's positivity theorem and random vector norms.
Operator theory, related fields, and applications, 149--215,
Oper. Theory Adv. Appl., 307, Birkhäuser/Springer, Cham, 2025.


\bibitem{BP}
Brazitikos, S., Pandis, C.,
Sharp inequalities for symmetric polynomials, Hunter's conjecture, and moments of exponential random variables.
Preprint (2025): arXiv:2512.12254.


\bibitem{Brz}
Brzezinski, P.,
Volume estimates for sections of certain convex bodies.
Math. Nachr. 286 (2013), no. 17-18, 1726--1743.



\bibitem{CGT}
Chasapis, G., Gurushankar, K., Tkocz, T., Sharp bounds on $p$-norms for sums of independent uniform random variables, $0 < p < 1$, 
J. Anal. Math. 149 (2023), no. 2, 529--553. 


\bibitem{CKT}
Chasapis, G., K\"onig, H., Tkocz, T.,
From Ball's cube slicing inequality to Khinchin-type inequalities for negative moments,
J. Funct. Anal. 281 (2021), no. 9, Paper No. 109185, 23 pp.


\bibitem{CNT}
Chasapis, G., Nayar, P., Tkocz, T.,  Slicing $\ell_p$-balls reloaded: stability, planar sections in $\ell_1$, Ann. Probab. 50 (2022), no. 6, 2344--2372.


\bibitem{CST0}
Chasapis, G., Singh, S., Tkocz, T.,
Entropies of sums of independent gamma random variables,  
J. Theoret. Probab. 36 (2023), no. 2, 1227--1242.


\bibitem{CST}
Chasapis, G., Singh, S., Tkocz, T.,
Haagerup's phase transition at polydisc slicing. Anal. PDE 17 (2024), no. 7, 2509–2539.

\bibitem{DDS}
De, A., Diakonikolas, I., Rocco A. S.,
A robust Khintchine inequality, and algorithms for computing optimal constants in Fourier analysis and high-dimensional geometry. 
SIAM J. Discrete Math. 30 (2016), no. 2, 1058--1094. 



\bibitem{Luc} 
Devroye, L., Nonuniform random variate generation, in Handbooks in operations research and management science, 13 (2006), 83--121.





\bibitem{EG}
Edwards, R. E., Gaudry, G. I., 
Littlewood-Paley and multiplier theory, Berlin, New York: Springer-Verlag, 1977.


\bibitem{ENT1}
Eskenazis, A., Nayar, P., Tkocz, T., Gaussian mixtures: entropy and geometric inequalities, Ann. of Prob. 46(5) 2018, 2908--2945.



\bibitem{ENT2}
Eskenazis, A., Nayar, P., Tkocz, T.,
Sharp comparison of moments and the log-concave moment problem,
Adv. Math. 334 (2018) 389--416.



\bibitem{Haa}
Haagerup, U.,
The best constants in the Khintchine inequality.
Studia Math. 70 (1981), no. 3, 231--283.

\bibitem{Hun}
Hunter, D. B.,
The positive-definiteness of the complete symmetric functions of even order. 
Math. Proc. Cambridge Philos. Soc. 82 (1977), no. 2, 255--258.

\bibitem{HNVW}
Hyt\"onen, T., van Neerven, J., Veraar, M., Weis, L., Analysis in Banach spaces. Vol. II. Martingales and Littlewood-Paley theory. Series of Modern Surveys in Mathematics  63. Springer, Cham, 2016.


\bibitem{KT}
Khare, A., Tao, T., 
On the sign patterns of entrywise positivity preservers in fixed dimension. 
Amer. J. Math. 143 (2021), no. 6, 1863--1929.

\bibitem{Khi}
Khintchine, A.,
\"Uber dyadische Br\"uche.
Math. Z. 18 (1923), no. 1, 109--116.



\bibitem{King}
Kingman, J. F. C., Poisson processes. Oxford Studies in Probability, 3. Oxford Science Publications. The Clarendon Press, Oxford University Press, New York, 1993.



\bibitem{KST}
K\"onig, H., Sch\"utt, C., Tomczak-Jaegermann, N.,
Projection constants of symmetric spaces and variants of Khintchine's inequality.
J. Reine Angew. Math. 511 (1999), 1--42.



\bibitem{LamT}
Lamkin, P., Tkocz, T., Log-concavity and log-convexity of moments of averages of i.i.d. random variables, \emph{Canad. Math. Bull.} 65 (2022), no. 2, 271--278.


\bibitem{LT}
Ledoux, M., Talagrand, M., 
Probability in Banach spaces. Isoperimetry and processes. 
Reprint of the 1991 edition. Classics in Mathematics. Springer-Verlag, Berlin, 2011


\bibitem{Mac}
Macdonald, I. G.,
Symmetric functions and Hall polynomials. Second edition. Oxford Classic Texts in the Physical Sciences. The Clarendon Press, Oxford University Press, New York, 2015.


\bibitem{MRTT}
Melbourne, J., Roysdon, M., Tang, C., Tkocz, T.,
From simplex slicing to sharp reverse Hölder inequalities, to appear in J. Lond. Math. Soc.,
Preprint (2025): arXiv:2505.00944.



\bibitem{MTTT}
Myroshnychenko, S., Tang, C., Tatarko, K., Tkocz, T., to appear in Discrete Comput. Geom.,
Stability of simplex slicing.
Preprint (2024): arXiv:2403.11994.



\bibitem{NT-surv}
Nayar, P., Tkocz, T., Extremal sections and projections of certain convex bodies: a survey. Harmonic analysis and convexity, 343--390, Adv. Anal. Geom., 9, De Gruyter, Berlin, 2023.

\bibitem{OP}
Oleszkiewicz, K., Pe\l czy\'nski, A.,
Polydisc slicing in $C^n$.
Studia Math. 142 (2000), no. 3, 281--294.




\bibitem{T}
Tang, C., 
Simplex slicing: An asymptotically-sharp lower bound,
Adv. Math. 451 (2024), 109784.


\bibitem{Tao-blog}
Tao, T., Schur convexity and positive definiteness of the even degree complete homogeneous symmetric polynomials, 2017. Blog post: \url{https://terrytao.wordpress.com/2017/08/06/}


\bibitem{Tao}
Tao, T.,
A variant of Maclaurin's inequality. 
Proc. Amer. Math. Soc. Ser. B 12 (2025), 1--13.




\bibitem{Webb}
Webb, S., 
Central slices of the regular simplex. 
Geom. Dedicata 61 (1996), no. 1, 19--28.

\bibitem{Wid}
 Widder, D. V., 
The Laplace Transform. Princeton Mathematical Series, vol. 6. Princeton University Press, Princeton, NJ, 1941.


\end{thebibliography}
\end{document}